\documentclass{amsart}[12pt]
\usepackage[all]{xy}
\usepackage{amsmath}
\usepackage{amsfonts}
\usepackage{amssymb}
\usepackage{amsthm}
\usepackage{enumerate}
\usepackage{amscd}
\usepackage{latexsym}
\usepackage{txfonts}
\usepackage{tikz}
\usepackage{graphicx}
\usepackage{float}
\usepackage[all]{xy}
\usepackage{hyperref}
\usepackage{color}

\newcommand{\bb}{\mathbb}

\newcommand{\ov}{\overline}
\theoremstyle{plain}
\newtheorem{theorem}{Theorem}
\newtheorem{lemma}[theorem]{Lemma}
\newtheorem{corollary}[theorem]{Corollary}

\newtheorem{conj}[theorem]{Conjecture}
\newtheorem{definition}[theorem]{Definition}
\newtheorem*{theorem*}{Theorem}
\theoremstyle{remark} 
\theoremstyle{remark} \newtheorem{example}{Example}
\theoremstyle{remark}

\numberwithin{theorem}{section}

\begin{document}

\title{Induced forests in some distance-regular graphs}

\author[K. Gunderson]{Karen Gunderson}
	\address{Department of Mathematics, University of Manitoba,	Winnipeg, 
	Manitoba R3T 2N2, Canada}\email{Karen.Gunderson@umanitoba.ca}
	\thanks{The first author was supported by Natural Science and Engineering Research Council of Canada (grant RGPIN-2016-05949).}
\author[K. Meagher]{Karen Meagher}
	\address{Department of Mathematics and Statistics, University of Regina,	Regina, 
	Saskatchewan S4S 0A2, Canada}\email{karen.meagher@uregina.ca}
	\thanks{The second author was supported by Natural Science and Engineering Research Council of Canada (grant RGPIN-03952-2018).}
\author[J. Morris]{Joy Morris}
	\address{Department of Mathematics and Computer Science, University of Lethbridge,	Lethbridge, 
	Alberta T1K 3M4, Canada}\email{joy.morris@uleth.ca}
		\thanks{The third and fourth authors were supported by the Natural Science and Engineering Research Council of Canada (grant RGPIN-2017-04905).}
\author[Venkata Raghu Tej Pantangi]{Venkata Raghu Tej Pantangi}
	\address{Department of Mathematics and Computer Science, University of Lethbridge,	Lethbridge, 
	Alberta T1K 3M4, Canada}\email{raghu.pantangi@uleth.ca}
	\thanks{The authors are all indebted to the support of the Pacific Institute for Mathematical Sciences (PIMS), through the establishment of the Collaborative Research Group on Movement and Symmetry in Graphs which funded this work.}

\date{\today}

\begin{abstract}
In this article, we study the order and structure of the largest induced forests in some families of graphs.   First we prove a variation of the ratio bound that gives an upper bound on the order of the largest induced forest in a graph. Next we define a \textsl{canonical induced forest} to be a forest that is formed by adding a vertex to a coclique and give several examples of graphs where the maximal forest is a canonical induced forest. These examples are all distance-regular graphs with the property that the Delsarte-Hoffman ratio bound for cocliques holds with equality.  We conclude with some examples of related graphs where there are induced forests that are larger than a canonical forest.
\end{abstract}

\subjclass[2010]{Primary: 05C69, Secondary: 05C35, 05C25}
\keywords{induced forests, distance-regular graphs, acyclic number, $1$-degenerate subgraphs}

\maketitle

\section{Introduction}

In this paper we study both the cardinality and structure of the largest sets of vertices inducing forests in some distance-regular graphs.  For a graph $G$, let $\tau(G)$ be the maximum number of vertices inducing a forest in $G$. The quantity $\tau(G)$ is called the \emph{acyclic number of $G$}. Letting $\alpha(G)$ denote the independence number of $G$, the order of the largest coclique, it is clear that for any non-empty graph, $\tau(G) \geq \alpha(G) +1$ as adding any vertex to an independent set will induce a forest.  The main results of this article are to give bounds on $\tau(G)$ for certain distance-regular graphs and to identify graphs in which every maximum induced forest can be obtained by adding a single vertex to an independent set.

A number of other graph parameters and special kinds of vertex subsets bear some relationship to this acyclic number $\tau(G)$. An induced forest in a graph is complementary to a set of vertices whose removal induces an acyclic graph and this is sometimes known as a `decycling set' of a graph, or a `feedback vertex set'.  Recall that a graph is $k$-degenerate if and only if every subgraph has a vertex of valency at most $k$.  The notion of degeneracy arises in colouring problems and in the study of `cores' of graphs, related to their connectivity properties.  A graph is empty if and only if it is 0-degenerate, while a graph is a forest if and only if it is $1$-degenerate.  Thus, the largest coclique in a graph is the largest set of vertices that induce a $0$-degenerate subgraph, while the largest induced forest can be thought of as the largest set of vertices inducing a $1$-degenerate subgraph.

Alon, Kahn, Seymour~\cite{AKS87} showed that $\tau(G) \geq \sum_{v\in V}2/(d(v)+1)$, where $d(v)$ denotes the valency of $v$. In fact, this is a special case of the general bound they prove for $k$-degenerate induced subgraphs. In the case of a $d$-regular graph on $n$ vertices, this implies that $\tau(G) \ge 2n/(d+1)$. This bound is tight when $(d+1)\mid n$ for a graph consisting of disjoint copies of $K_{d+1}$.   Bondy, Hopkins and Staton~\cite{BHS87} showed that if $d=3$ and $G$ is connected (so that the previous tight examples do not apply), then $\tau(G) \geq \frac{5n-2}{8}$ (here $n$ is the number of vertices). They also provided examples where their bound is tight.  Further refinements have been given for regular graphs of large girth~\cite{HW08, KL18, LZ96}.   Bau, Wormald, and Zhou~\cite{BWZ02} showed that for random $3$-regular graphs, asymptotically almost surely, $\tau(G) = n-\lceil (n+2)/4 \rceil=\lfloor (3n-2)/4\rfloor$  and gave bounds for random $r$-regular graphs in general. Alon, Mubayi and Thomas~\cite{AMT01} gave bounds on $\tau(G)$ in terms of the independence number and the maximum valency.

The largest induced forests and smallest decycling sets in specific families of graphs have been well-studied in the literature, for example: planar graphs~\cite{AW87}, bipartite graphs~\cite{nA03, CFS23}, hypercubes~~\cite{BBDLV01, FLP00, dP03} and  binomial random graphs~\cite{KZ21}. Related work has concerned the largest induced trees~\cite{ESS86, FLS09, MS08, PR86, fP10} and the largest induced matchings~\cite{kC89, CDKS21}.

One of the most well known results in extremal graph theory is the Erd\H{o}s-Ko-Rado theorem (\cite{EKR61}). 

\begin{theorem}[Erd\H{o}s, Ko, Rado]
Let $n > 2k$ and let $\mathcal{F} = \{F_1, F_2, \ldots, F_m\}$ be an intersecting family of $k$-sets from $[n]$.  Then,
\[
|\mathcal{F}| \leq \binom{n-1}{k-1},
\]
with equality if and only if $\mathcal{F}$ consists of all $k$-sets containing a fixed element $x \in [n]$.
\end{theorem}

This celebrated result can be interpreted as a characterization of the cardinality and structure of independent sets of maximum possible cardinality in the Kneser graph.  The Kneser graph $K(n, k)$ is defined for any $n, k \in \mathbb{Z}^+$ to be the graph whose vertices are all of the $k$-sets from $[n] = \{1, 2, \ldots, n\}$ with two vertices $A, B$ adjacent if and only if $A \cap B = \emptyset$. If $n <2k$, then $K(n,k)$ has no edges, so we assume that $n \geq 2k$. Translated into the setting of the Kneser graphs, the Erd\H{o}s-Ko-Rado Theorem states that for $n > 2k$, $\alpha(K(n, k)) = \binom{n-1}{k-1}$ and any coclique of this order consists of all $k$-sets that contain a common element.   

Similar such characterizations were made for maximum cocliques in many other families of graphs. We refer to \cite{GMbook} for a through survey of such results. 
The main results of this article characterize the largest induced forests in some distance-regular graphs. The graphs we consider are distance-regular graphs for which the characterization of maximum independent sets is known. Throughout this article, we will refer to induced forests of maximum possible order as \textsl{maximum induced forests}.
Let $G=(V,\ E)$ be a graph and $S$ be a coclique of $V$. As noted previously, for any $v \in V \setminus S$, the set $S \cup \{v\}$ induces a forest, so that $\tau(G) \geq \alpha(G)+1$. 
A natural next step is to find graphs in which every maximum induced forest can be constructed by adding a vertex to an independent set. 

\begin{definition}\label{Nat}
Let $G=(V,\ E)$ be a graph and let ${F}\subset V$ induce a forest. We say that ${F}$ is a \textsl{canonical} induced forest if there is a vertex $v \in {F}$ such that ${F} \setminus \{v\}$ is an independent set. Often we refer to this as just a canonical forest in $G$.
\end{definition}   

The following result, simply known as the Delsarte-Hoffman ratio bound, is a spectral graph theoretic method that has been used to characterize the maximum cocliques in many families of graphs. 

\begin{theorem}\label{ratiobound}(see \cite[Theorem 3.2]{Hae1995})
Let $G$ be a $k$-regular graph on $n$ vertices and let $\lambda$ be the smallest eigenvalue of the adjacency matrix of $G$. Then we have 
\[
\alpha(G) \leq \dfrac{n(-\lambda)}{k-\lambda}.
\]
\end{theorem}   

This result is an application of the Cauchy Interlacing Theorem (see \cite[Theorem 2.1]{Hae1995}). Applying the same technique, we will show the following spectral upper bound for the order of an induced forest in a regular graph.

\begin{theorem}\label{algbound}
Let $G$ be a $k$-regular graph on $n$ vertices and let $\lambda$ be the smallest eigenvalue of the adjacency matrix of $G$. Then
\[
\tau(G) \leq \frac{n(2-\lambda) + \sqrt{n^{2}(2-\lambda)^2-8n(k-\lambda)}}{2(k-\lambda)} < \frac{-n\lambda}{k-\lambda} +\frac{2n}{k-\lambda}.
\]
\end{theorem}

An edge-counting argument provides an alternative bound on the order of an induced forest in a regular graph that is sometimes better than the spectral bound (see discussion after Lemma~\ref{lem:2k+1forest}).

\begin{theorem}\label{elembound}
Let $G$ be a $k$-regular graph on $n$ vertices. Let $f$ be the number of vertices and $c$ the number of connected components in an induced forest of $G$. Then
$$f \le \frac{nk-2c}{2k-2} \le \frac{nk-2}{2k-2}.$$
\end{theorem}

 The first summand of the right-hand side of the inequality in Theorem~\ref{algbound} is equal to the Delsarte-Hoffman ratio bound on the independence number $\alpha(G)$. It is natural to investigate the orders of forests in regular graphs for which the Delsarte-Hoffman ratio bound is tight. Below is a list of five families of such graphs in which the maximum forest is formed by adding a single vertex to a coclique.

\begin{theorem}
\label{thm:examples}
In the following graphs, every maximum forest is a canonical forest:
\begin{enumerate}
\item the Kneser graph $K(n, k)$, for every $k \geq 2$ and $n \geq 2k^3$;
\item the $q$-Kneser graph $K_q(n, k)$, for $k\geq 2$, $n>3k-2$ and $q$ sufficiently large;
\item the non-collinearity graph on points in a generalized quadrangle with parameters $(s, t)$ and $s>3$;
 \item $X_{m,n} = \otimes^{m}K_{n}$ with $m\ge 2$ and $n>2m(m-1)$;
 \item the complement of the block graph of an orthogonal array with parameters $m, n$ with $n>1+2m(m-1)$;
\end{enumerate}
\end{theorem}

We were able to make a few refinements in some subfamilies of the graphs mentioned in the above result. These can be found in Theorems~\ref{thm:Kneserk2}, \ref{cor:qKneserk2}, and \ref{TensorPsrg}. 

We prove Theorems~\ref{algbound} and~\ref{elembound} in Section~\ref{UB}. In Section~\ref{Can}, we prove the results of Theorem~\ref{thm:examples}, characterizing induced forests in some other families of graphs. In Section~\ref{large}, we produce an infinite family of graphs with ``large'' maximum forests.

\section{Upper bounds}\label{UB} 

We begin this section by proving Theorem~\ref{elembound}. 

\begin{proof}[Proof of Theorem~\ref{elembound}] 
Let $G=(V,\ E)$ be a $k$-regular graph on $n$ vertices, and let $F$ be an induced forest of $G$ with $f$ vertices and $c$ connected components. 

Now, $F$ has $f-c$ edges. Since each of the $f$ vertices of $F$ has $k$ incident edges and each of the $f-c$ edges of $F$ is counted twice in the valency of vertices of $F$, there are $fk-2(f-c)=f(k-2)+2c$ edges of $G$ that join vertices of $F$ to vertices that are not in $F$. In total, this makes $f(k-1)+c$ edges of $G$ that are incident with at least one vertex of $F$. 

Clearly, the number of edges of $G$ that are incident with at least one vertex of $F$ cannot exceed the total number of edges of $G$, which by the Handshaking Lemma is $nk/2$. So
$$f(k-1)+c \le nk/2.$$
Rearranging this inequality produces the given result, which is maximized when $c=1$.
\end{proof}

We next work toward the proof of Theorem~\ref{algbound}.
Let $G=(V,\ E)$ be a $k$-regular graph on $n$ vertices. Let $k=\lambda_{1}\geq \lambda_{2}\cdots\geq \lambda_{n}$ be the eigenvalues of its adjacency matrix. The following result from \cite{Hae1995} gives algebraic bounds for induced subgraphs. We include the proof for completeness.

\begin{theorem}\cite[Theorem 3.5]{Hae1995}
Let $G$ be a $k$-regular graph on $n$ vertices and suppose that $G$ has an induced subgraph $G'$ with $n'$ vertices and $m'$ edges. Then 
\[
\lambda_{2} \geq \frac{2m'n-k (n')^2}{n'(n-n')} \geq \lambda_{n}.
\]
\end{theorem}
\begin{proof}
Consider the partition $\pi=\{G',\ \overline{G'}\}$ of the vertex set. The corresponding quotient matrix is 
\[
\begin{pmatrix}
\frac{2m'}{n'} & k -\frac{2m'}{n'}\\
\frac{n'k-2m'}{n-n'} & k-\frac{n'k-2m'}{n-n'}
\end{pmatrix}.
\]
The eigenvalues of this matrix are $k$ and $\frac{2m'}{n'} - \frac{n'k-2m'}{n-n'} = \frac{2m'n-n'^2k}{n'(n-n')}$.
The result follows by Cauchy's Interlacing Theorem (see \cite[Theorem 2.1]{Hae1995}).
\end{proof}

We are now ready to prove Theorem~\ref{algbound}. 

\begin{proof}[Proof of Theorem~\ref{algbound}]
Let $F$ be an induced forest in $G$ on $f$ vertices with $c$ connected components. 
Since $F$ has exactly $f-c$ edges and $f$ vertices, using the above result, we have 
\[
\frac{2(f-c)n-f^{2}k}{f(n-f)} \geq \lambda_{n},
\]
and thus 
\[
(k-\lambda_{n})f^{2}+n(\lambda_{n}-2)f+2cn\leq 0.
\]
As $c\geq 1$, we have $(k-\lambda_{n})f^{2}+n(\lambda_{n}-2)f +2n\leq 0$, and thus $f \leq \frac{n(2-\lambda_{n}) + \sqrt{n^{2}(2-\lambda_{n})^2-8n(k-\lambda_{n})}}{2(k-\lambda_{n})}$. 
\end{proof}

We now use Theorem~\ref{algbound} to find the acyclic number of a small graph.
\begin{example}\label{exampleP9}
Consider the complement $\mathcal{P'}(9)$ of the Paley graph on $9$ vertices. The vertex set of this graph is the field $\bb{F}_{9}$ of size $9$; and two elements $a,b \in \bb{F}_{9}$ are adjacent if and only if $a-b$ is not a quadratic residue in $\bb{F}_{9}$. We identify $\bb{F}_{9} \cong \bb{F}_{3}[x]/\langle x^{2}+1\rangle$ and the set of quadratic residues is $S=\{\ov{0},\ \ov{1},\ \ov{2},\ \ov{x},\ \ov{2x}\}$. The induced subgraph $\bb{F}_3 \cup \{ \ov{x+1},\ \ov{x+2} \}$ is a path on 5 vertices, in $\mathcal{P'}(9)$. This construction implies that $\tau\left(\mathcal{P'}(9)\right) \geq  5$.  It is well-known that $\mathcal{P'}(9)$ is a strongly-regular graph whose specturm is $(4,\ 1,\ -2)$. Using Theorem~\ref{algbound}, we have $\tau(\mathcal{P'}(9)) < 6$. We note that Theorem~\ref{elembound} gives us the same upper bound. Thus we have $\tau(\mathcal{P'}(9)) = 5$. 
\end{example}

We were not able to extend this to other Paley graphs. In Section~\ref{large}, we present some observations (on the acyclic number) stemming from computations on small order Paley graphs.

\section{Graphs whose maximum induced forests are canonical.}\label{Can}

In this section, we characterize maximum induced forests in some families of regular graphs. In particular, we will prove Theorem~\ref{thm:examples} using a counting method for each graph.

Let $G$ be a regular graph. We recall that the order $\tau(G)$ of a maximum induced forest satisfies $\tau(G) \geq \alpha(G)+1$. To show that every maximum induced forest in $G$ is canonical, it suffices to show that $|F| < \alpha(G)+1$ for every non-canonical induced forest $F$. 
Note that an induced forest $F$ in $G$ is not canonical if and only if $F$ contains either a copy of $P_{4}$ (a path with $4$ vertices) or a copy of $P_{2}+P_{2}$ (the disjoint union of two edges) as an induced subgraph. We now find an upper bound on the order of an induced forest $F$ that does not contain either a $P_{4}$ or a $P_{2}+P_{2}$.

Given a pair $(a, b)$ of adjacent vertices in $G$, by $N(a,b)$, we denote the set of vertices in $G$ that are not adjacent to either of $a$ or $b$; and by $\eta(a,b)$, we denote $|N(a,b)|$.  We denote the maximum such value by 
\[
\eta(G) = \max \left \{\eta(a,b)\ |\ a,b\in G \text{ and } a\sim b \right \} . 
\]

\begin{lemma}\label{lem:count}
If $F$ is a non-canonical forest in a graph $G$, then $|F| \leq 2+ 2 \eta(G)$.
\end{lemma}

\begin{proof}
First assume that $F$ contains a path on four vertices; call this subgraph $P$. Since $F$ is a forest, every $v \in F \setminus P$ is adjacent to at most one vertex of $P$. Therefore, every $v \in F \setminus P$ is non-adjacent to at least one leaf and the neighbour of that leaf in $P$. Suppose that $P$ is made up of vertices $\{a,b,c,d\}$ with $a\sim b$, $b \sim c$ and $c \sim d$. Then we see that $F \subset N(a,b) \cup N(c,d) \cup \{b,c\}$, completing the proof in this case.

Next assume that $F$ does not contain a path on four vertices but has an induced subgraph $Q$ that is isomorphic to $P_{2}+P_{2}$. Let $Q$ be made up of vertices $\{a,b,c,d\}$ such that $a\sim b$ and $c \sim d$. Since $F$ is a forest that does not contain a path on four vertices, every $v \in F \setminus Q$ is adjacent to at most one vertex of $Q$, so is non-adjacent to a pair of adjacent vertices of $Q$. Thus we have $F \subseteq N(a,b) \cup N(c,d)$. 
\end{proof}

This lemma is particularly applicable to strongly-regular graphs since the value of $\eta(\alpha,\beta)$ is the same for all pairs $(\alpha,\beta)$ of adjacent vertices. We now recall that given $n,k,a,c \in \bb{N}$, a strongly-regular graph with parameters $(n, k: a, c)$ is a $k$-regular graph on $n$ vertices such that (i) every pair of adjacent vertices have exactly $a$ neighbours in common; and (ii) every pair of non-adjacent vertices have exactly $c$ neighbours in common. Using inclusion-exclusion on the parameters of a strongly-regular graph to get the value of $\eta(\alpha,\ \beta)$ yields the following result.

\begin{corollary}
Let $G$ be a strongly-regular graph with parameters $(n, k: a, c)$. If
\[
1+2(n-2k+a) < \alpha(G),
\] 
then every maximum induced forest is a canonical induced forest.
\end{corollary}

This corollary can be used to prove that for $n \geq 17$ the maximum forests in $K(n,2)$ are canonical (we omit this proof, since Theorem~\ref{thm:Kneserk2} gives a stronger result).

In the following subsections, we apply Lemma~\ref{lem:count} to show that maximum induced forests in some families of graphs must be canonical.

\subsection{Kneser Graphs}
\label{subsec:kneser}
 
In this section we consider the Kneser graphs $K(n,k)$ with $n \geq 2k$. The graph $K(2k, k)$ consists of exactly $\frac12\binom{2k}{k}$ disjoint edges and is itself a forest, so we will only consider $n>2k$. It is well known from the Erd\H{o}s-Ko-Rado Theorem~\cite{EKR61} that the order of a maximum coclique in $K(n,k)$ is ${n-1 \choose k-1}$ and that the Delsarte-Hoffman ratio bound holds with equality. Thus a canonical forest has order ${n-1 \choose k-1} + 1$. We will show for $n$ large relative to $k$ that this is the largest possible induced forest.

\begin{theorem}\label{Kneser}
For every $k \geq 2$ and $n \geq 2k^3$, we have
\[
\tau(K(n, k)) = \binom{n-1}{k-1} + 1.
\] 
Moreover, every maximum induced forest is a canonical induced forest. 
\end{theorem}
\begin{proof}
Let $\gamma$ and $\delta$ be a pair of adjacent vertices in $K(n, k)$. Elementary counting arguments (overcounting sets whose intersection with $\gamma$ or $\delta$ has cardinality greater than $1$) show that there are at most $k^{2}{n-2 \choose k-2}$ $k$-subsets of $[n]$ intersecting both $\gamma$ and $\delta$. Thus in this case, we have $\eta(K(n,k))\le k^{2}{n-2 \choose k-2}$. 

By Lemma~\ref{lem:count}, if $F$ is a non-canonical induced forest, then $|F| \leq 2 + 2 k^{2}{n-2 \choose k-2}$. In the case $n\geq 2k^{3}$, we have \[
2 + 2 k^{2}{n-2 \choose k-2} < 1+{n-1 \ \choose k-1}.
\] 
Therefore non-canonical induced forests are smaller than the canonical induced forests.
\end{proof}

We consider one special case of Kneser graphs, in which the same sort of counting can be done more precisely.

\begin{theorem}\label{thm:Kneserk2}
For $n \geq 5$
\[
\tau(K(n, 2)) = \max \{ n, 7 \}.
\]
If $n > 7$, every maximum induced forest in $K(n,2)$ is canonical.
\end{theorem}
\begin{proof}
For any $n$, a canonical forest in $K(n,2)$ has order $n$. Further, $\tau(K(n,2)) \geq 7$ for any $n \geq 5$, this is seen by taking the following vertex set:
\[
\left \{ \{1,2\},\{3,4\},  \{1,3\},\{2,4\},  \{1,4\},\{2,3\}, \{1,5\}   \right \}.
\]

Recall that any non-canonical forest contains either a copy of $P_4$ or a copy of $P_2+P_2$.

Assume $F$ is an induced forest in $K(n, 2)$. If $F$ contains a copy of $P_4$, then the vertices of this $P_4$ must be the sets $\{a,b\}, \{c,d\}, \{a,e\},\{b,c\}$ for some $a,b,c,d,e$. Any other vertex in $F$ is 
adjacent to at most one of these vertices. There are only $6$ other vertices in $K(n,2)$ that are nonadjacent to any of the pairs of adjacent vertices on this path, and exactly $3$ of these vertices are nonadjacent to $3$ vertices of the path ($\{a,c\}, \{b,d\},$ and $\{b,e\}$). So any such $F$ contains at most $7$ vertices. 

Similarly, if $F$ contains a copy of $P_2 + P_2$, then the vertices of this subgraph must be the sets $\{a,b\}, \{c,d\},\{a,c\},\{b,d\}$ for some $a,b,c,d$. If $F$ has no $P_4$, then $F$ cannot include any vertex of the form $\{a,e\},\{b,e\},\{c,e\},$ or $\{d,e\}$ (for any $e \notin \{a,b,c,d\}$). Since any vertex in $F$ must be nonadjacent to at least $2$ of the vertices of the $P_2 +P_2$, this implies that the elements of the $2$-set defining the vertex must lie entirely in $\{a,b,c,d\}$, so there are only $2$ other vertices that can be added: $\{a,d\}$ and $\{b,c\}$. So any such $F$ contains at most $6$ vertices.

Therefore any induced forest that is not canonical contains no more than $7$ vertices and the result follows.
\end{proof}

\subsection{q-Kneser Graphs}

The next family we consider is the $q$-Kneser graphs. Let $n, k$ be positive integers with $n\geq 2k$, and $q$ be a power of a prime. The vertex set of the graph $K_{q}(n, k)$ is the set of all $k$-dimensional subspaces of $\bb{F}_{q}^{n}$; two vertices are adjacent if and only if they intersect trivially. It is well known that the cardinality of a coclique in this graph is ${n-1 \choose k-1}_{q}$, and that the Delsarte-Hoffman ratio bound holds with equality (see \cite{FW86} or \cite[Chapter 9]{GMbook} for notation and details). The canonical induced forests have $1+ {n-1 \choose k-1}_{q}$ vertices. We obtain the following characterization of maximum induced forests in $q$-Kneser graphs.

\begin{theorem}\label{q-Kneser}
For $k\geq 2$, $n>3k-2$ and $q$ sufficiently large, we have
\[
\tau(K_{q}(n, k)) = \binom{n-1}{k-1}_{q} + 1.
\] 
Moreover, every maximum induced forest is canonical.
\end{theorem}
\begin{proof}
Let $\gamma$ and $\delta$ be two adjacent vertices in $K(n, k)_q$. If $\omega$ is a $k$-subspace intersecting non-trivially with both $\gamma$ and $\delta$, then it contains a subspace of the form $\langle x \rangle+\langle y \rangle$, where $x \in \gamma \setminus \{0\}$ and $y \in \delta \setminus \{0\}$. A subspace of the form $\langle x \rangle+\langle y \rangle$ can be chosen in ${k \choose 1}_{q}^{2}$ ways. It is a well known fact that there are ${n-2 \choose k-2}_{q}$ subspaces of dimension $k$, which contain a specific $2$-dimensional subspace. Thus we have $\eta(K_{q}(n,k)) \leq  {k \choose 1}_{q}^{2}{n-2 \choose k-2}_{q}.$

If $F$ is a non-canonical induced forest, then by Lemma~\ref{lem:count}, we have 
\[
|{F}| \leq 2 + 2{k \choose 1}_{q}^{2}{n-2 \choose k-2}_{q}.
\]

We will now show that, provided $n>3k-2$ and $q$ sufficiently large,  this upper bound is smaller than $1+ \binom{n-1}{k-1}_{q}$.
Since ${n-1 \choose k-1}_{q}= \dfrac{{n-1 \choose 1}_{q}}{{k-1 \choose 1}_{q}}{n-2 \choose k-2}_{q}$, we have  
\begin{align*}
1+{n-1 \choose k-1}_{q} -2- 2{k \choose 1}_{q}^{2}{n-2 \choose k-2}_{q} &= {n-2 \choose k-2}_{q} \left( \dfrac{{n-1 \choose 1}_{q} - 2{k-1 \choose 1}_{q}{k \choose 1}_{q}^{2}}{{k-1 \choose 1}_{q}} \right)-1.  
\end{align*}
Expanding the $q$-binomial coefficients gives that
\[
{n-1 \choose 1}_{q}  - 2{k-1 \choose 1}_{q}{k \choose 1}_{q}^{2} 
 = 
  \frac{q^{n-1} -1}{q-1} - 2 \left( \frac{q^{k-1}-1}{q-1} \right) \left( \frac{q^k-1}{q-1}\right)^2 ,
\]
and, provided that $n-2>3k-4$, this is a monic polynomial of degree $n-2$ and hence positive for a sufficiently large $q$. 
So for $n>3k-2$ and $q$ sufficiently large, the order of any forest is bounded above by $1+ {n-1 \choose k-1}_{q}$, and this bound is met by only canonical forests. 
\end{proof}

As in the case of Kneser graphs, we consider the special case of strongly-regular $q$-Kneser graphs with $k=2$, in which the same sort of counting can be done more precisely. In particular, the following result gives a complete characterization of maximum forests in $K_{q}(n,2)$  provided $n\geq 4$.

\begin{theorem}\label{cor:qKneserk2}
For $n \geq 4$
\[
\tau(K_{q}(n, 2)) = \max \left\{ {n-1 \choose 1}_{q}+1, 8 \right\}.
\]
If $(n,q)\neq(4,2)$, then every maximum induced forest in $K(n,2)$ is canonical.
\end{theorem}
\begin{proof}
Let $F$ be a non-canonical forest in $K_{q}(n,2)$. Then $F$ contains either a copy of $P_{4}$ (a path with $4$ vertices) or a copy of $P_{2}+P_{2}$ (the disjoint union of two edges) as an induced subgraph. 

First assume that $F$ has four vertices $\{X,Y,V,W\}$ inducing a path, with $X\sim Y$, $Y \sim V$, and $V \sim W$. From the discussion prior to Lemma~\ref{lem:count}, we have $F \subset N(X,Y)  \cup N(V,W) \cup \{Y,V\}$. As $K_{q}(n,2)$ is an strongly-regular graph, the graph induced by $N(X,Y)$ is isomorphic to the graph induced by $N(V,W)$. We now have $|F| \leq 2+ 2\tau(N(V,W))$, where $\tau(N(V,W))$ is the order of a maximum forest induced in the graph $N(V,W)$. Similarly, if $F$ has four vertices $\{X,Y,V,W\}$ inducing a disjoint union of two edges, with $X\sim Y$ and $V \sim W$, then $|F| \leq 2\tau(N(V,W))$. Therefore the order of a non-canonical forest is bounded above by $2+ 2\tau(N(V,W))$. We will now try to look at the structure of $N(V,W)$.

As $V$ and $W$ are adjacent, $V$ and $W$ are disjoint $2$-subpaces of $\bb{F}_{q}^{4}$, and thus any $U\in N(V,W)$ is completely determined by $U \cap V$ and $U \cap W$. Let $U_{1},U_{2} \in N(V,W)$ intersect non-trivially. Let $e_{1},e_{2} \in V$ and $f_{1},f_{2}\in W$ be such that $e_{i}\in U_{i}\cap V$ and $f_{i}\in U_{i} \cap W$. As $U_{1}$ and $U_{2}$ intersect non-trivially, there are $a,b,c,d \in \bb{F}_{q}$ such that $ae_{1}+bf_{1}=ce_{2}+df_{2}$. This can be rewritten as $ae_{1}-ce_{2}=df_{2}-bf_{1}$. As $V$ and $W$ are disjoint, we must have $ae_{1}=ce_{2}$ and $df_{2}=bf_{1}$. We have now concluded that for $U_{1}, U_{2} \in N(V,W)$, $U_{1} \sim U_{2}$ if and only if $U_{1}\cap V \neq U_{2} \cap V$ and $U_{1}\cap W \neq U_{2} \cap W$. This shows that the subgraph induced by $N(V,W)$ is the two fold tensor product of the complete graph $K_{q+1}$. We dealt with these graphs in Subsection~\ref{SubSec:TP}. Using the notation in Subsection~\ref{SubSec:TP}, we have $N(V,W) \cong X_{2,q+1}$.

In Theorem~\ref{TensorP} we will show that provided $q\geq3$, we have $\tau(X_{2,q+1})=q+2$. Therefore if $q\geq 3$, the order of a non-canonical forest is bounded above by $2+2\tau(N(V,W))=2q+6$. The order of the largest canonical forest is $\alpha(K_{q}(n,2))+1 ={n-1 \choose 1}_{q}+1$. Elementary algebra shows that ${n-1 \choose 1}_{q}+1>2q+6$ for all $n\geq 4$ and $q\geq 3$. Therefore provided $q\geq 3$, every maximum forest in $K_{q}(n,2)$ is canonical.

We now shift our attention to $q=2$. If $V,W$ are two adjacent vertices in $K_{2}(n,2)$, then we have seen that $N(V,W) \cong X_{2,3}= K_{3} \otimes K_{3}$. The algebraic bound Theorem~\ref{algbound} shows that $\tau(X_{2,3})< 6$. We can check either by hand or computer that $X_{2,3}$ has an induced path with $5$ vertices, and therefore $\tau(X_{2,3})=5$. Thus the order of a non-canonical forest $F$ is bounded above by $2+2\tau(N(V,W))=12$. For $n>4$, we have $\alpha(K_{2}(n,2))+1=2^{n-1}> 12$. Therefore provided $n>4$, every maximum forest in $K_{2}(n,2)$ is canonical.
 
We are now left with the case of $K_{2}(4,2)$. With the help of a computer algebra system such as Sage (\cite{sage}), we can show that $\tau(K_{2}(4,2))=\alpha(K_{2}(4,2))+1=8$. It can also be shown that there are paths on $8$ vertices in $K_{2}(4,2)$. Thus not all maximum forests are canonical in this case.

\end{proof}

\subsection{Non-collinearity Graphs of Generalized Quadrangles}

The next family we consider is the family of non-collinearity graphs on generalized quadrangles. Let $\mathcal{G}$ be a generalized quadrangle with parameters $s, t$. By $X_{\mathcal{G}}$, we denote the graph whose vertices are the points of $\mathcal{G}$, in which two points are adjacent if and only if they are not collinear. It is well known that $X_{G}$ is a strongly-regular graph with  $\left\{s^2t,\ -s,\ t \right\}$ as the set of distinct eigenvalues (see ~\cite[Section 5.6]{GMbook}). By the Delsarte-Hoffman ratio bound for cocliques (Theorem~\ref{ratiobound}), 
\[
\alpha(X_{G}) \leq \frac{(s+1)(st+1)s}{s^{2}t+s}=s+1.
\]
The set of all points on a line form a coclique, so this bound is tight.
We obtain the following characterization of maximum induced forests in $X_{\mathcal{G}}$.

\begin{theorem}\label{GQ}
Let $\mathcal{G}$ be a generalized quadrangle with parameters $(s, t)$ and let $X_{\mathcal{G}}$ be the non-collinearity graph on points in $\mathcal{G}$. Suppose that $s>3$, then,
\[
\tau(X_{\mathcal{G}})=s+2.
\]
Moreover, every maximum induced forest in $X_{\mathcal{G}}$ is canonical. 
\end{theorem} 
\begin{proof}

Consider an induced forest $F$ which contains a path $\mathcal{P}$ on $4$ vertices as an induced subgraph. Let $\{A,\ B,\ C,\ D\}$ be the vertices inducing $\mathcal{P}$, with $A \sim B$, $B \sim C$, and  $C \sim D$.
As $F$ is a forest, any $V \in F \setminus \mathcal{P}$ must be non-adjacent to at least three vertices in $\{A,\ B,\ C,\ D\}$. Suppose $V$ is non-adjacent to each vertex in $\{A,C,D\}$. In other words, $V$ is collinear with every point in $\{A,C,D\}$. Thus $V$ must lie on both the lines $\overrightarrow{AC}$ and $\overrightarrow{AD}$. This implies that $V=A$, which is contrary to our assumption $V \in F \setminus \mathcal{P}$. By the same argument, $V$ cannot be simultaneously non-adjacent to every vertex in $\{A,B,D\}$. Thus $V$ must be adjacent to one of $A$ or $D$. If $V$  is adjacent to $A$, then as $F$ is a forest, $V$ must be collinear to every point in $\{B,C,D\}$. Similarly, if $V$ is adjacent to $D$, then $V$ must be collinear to every point in $\{A,B,C\}$. As $\mathcal{G}$ is a generalized quadrangle, given a line $L$ and a point $P$ not on $L$, there is a unique point on $L$, that is collinear with $P$. Let $Q_{1}$ be the unique point on $\overrightarrow{BD}$ that is collinear with $C$, and let $Q_{2}$ be the unique point on $\overrightarrow{AC}$ that is collinear with $B$. We can now conclude that $F \subseteq \{Q_{1},\ A,\ B,\ C,\ D,\ Q_{2}\}$. We now claim that $Q_{1}$ and $Q_{2}$ are non-collinear. Let us assume the contrary, then we see that $\{Q_{1},\ C,\ Q_{2}\}$ form a triangle (not in the graph) in $\mathcal{G}$. This is impossible as a generalized quadrangle cannot contain a triangle, and therefore $Q_{1}$ and $Q_{2}$ are not collinear. Thus $S:=\{Q_{1},\ A,\ B,\ C,\ D,\ Q_{2}\}$ induces a cycle on $6$ vertices. As $F \subset \{Q_{1},\ A,\ B,\ C,\ D,\ Q_{2}\}$ is a forest, we must have $|F|\leq 5$.

Now consider an induced forest $F$ that contains a copy of $P_2 + P_2$. Let $\{P, Q\}$ and $\{R, S\}$ be two edges in distinct connected components of the forest. The points $P, R, Q, S$ form vertices of a quadrilateral in $\mathcal{G}$. Suppose that $|F|>4$, then any $V \in F \setminus \{P,\ R,\ Q,\ S\}$ must be non-adjacent to at least one point in both $\{P,\ Q\}$ and $\{R,\ S\}$. Without loss of generality, let $V$ be non-adjacent with $R$ and $Q$. We claim that $V$ must be on the line $\overrightarrow{RQ}$. Assuming the contrary implies the existence of the triangle $VRQ$ in $\mathcal{G}$, which is absurd as $\mathcal{G}$ is a generalized quadrangle. Again since  $\mathcal{G}$ is a generalized quadrangle, $R$ is the unique point on $\overrightarrow{RQ}$ collinear with $P$; and $Q$ is the unique point on $\overrightarrow{RQ}$ collinear with $S$. Therefore $V\in \overrightarrow{RQ}$, must be simultaneously non-collinear with both $P$ and $S$. Now the set $\{R,P,V,S\}$ induces a path on four vertices in $F$. By the argument in the above paragraph, existence of such a path implies that $|F|\leq 5$.   

From the previous two paragraphs, we can conclude that the size of a non-canonical forest is at most $5$. Since when $s>3$ any canonical forest has $s+2 \ge 6$ vertices, we have shown that if $s>3$, the only maximum forests in $X_{\mathcal{G}}$ are the canonical ones.
\end{proof}

\subsection{Tensor powers of complete graphs}\label{SubSec:TP}

We next consider a family of graphs in the Hamming scheme. Consider the complete graph on $n$ vertices, $K_{n}$. By $X_{m, n}$, we denote the $m$-fold tensor product $\otimes^{m}K_{n}$. This is the $m$th graph in the Hamming Scheme $H(m,n)$. The vertex set can be considered as sequences of length $m$ with entries from the additive group $\bb{Z}_{n}$, with two sequences adjacent if and only if they differ at every coordinate. This is an $(n-1)^{m}$-regular graph whose smallest eigenvalue is $-(n-1)^{m-1}$. Application of the Delsarte-Hoffman ratio bound (Theorem~\ref{ratiobound}) shows that $\alpha(X_{m,n})\leq n^{m-1}$. This bound is met by the subset of sequences whose first coordinate is $0$.

If $m=1$, then $X_{m,n} = K_n$ and any maximum forest is an edge which is a canonical maximum forest. 
Also, if $n=1$ then $X_{m,n}$ is simply $K_1$, so trivially any maximum forest is canonical.

We obtain the following characterization of maximum induced forests in $X_{m, n}$.

\begin{theorem}\label{TensorP}
Let $m, n$ be positive integers with $m \geq 2$ and $n>2m(m-1)$. Then
\[
\tau(X_{m,n})=n^{m-1}+1,
\]
and every maximum induced forest in $X_{m, n}$ is canonical. 
\end{theorem}
\begin{proof}
As before, we investigate the orders of non-canonical forests. 
A simple counting argument shows that  $\eta = m(m-1)n^{m-2}$ and therefore by Lemma~\ref{lem:count} a non-canonical forest has order at most $2 + 2 m(m-1)n^{m-2}$.
Thus canonical forests are the largest, provided that
\[
2 + 2 m(m-1)n^{m-2} < n^{m-1}+1,
\]
or, equivalently, 
\[
1 <n^{m-2} ( n  -    2 m(m-1) ) . 
\]
If $m \geq 2$, then the above equation holds whenever $n > 2m(m-1)$.
\end{proof}

As in the case of the $q$-Kneser graphs and the Kneser graphs, we consider the special case of strongly regular tensor powers of complete graphs, in which the same sort of counting can be done more precisely. In particular, the following result gives a complete characterization of maximum forests in $X_{2,n}$  provided $n\geq 4$.

\begin{theorem}\label{TensorPsrg}
Given $n\geq 3$, we have $\tau(X_{2,n})=max(\{5,\ n+1\})$. Moreover when $n\geq 4$, every maximum induced forest is canonical.
\end{theorem}  
\begin{proof}
Firstly given an edge $\{A,B\}$, we observe that $|N(A,B)|=2$, where $N(A,B)$ is the set of vertices that are not adjacent to either $A$ or $B$. Suppose that $F$ is a forest, with an induced subgraph $P \cong P_{4}$. Suppose that $P$ is made up of vertices $\{A,B,C,D\}$ with $A\sim B$, $B \sim C$ and $C \sim D$. Then from the discussion prior to Lemma~\ref{lem:count}, we know that $F \subset N(A,B) \cup N(C,D)\cup \{B,C\}$. Suppose that $X$ and $Y$ are vertices such that $N(A,B)=\{D,X\}$ and $N(C,D)=\{A,Y\}$. Without loss of generality, we may assume that $A=(a,b)$, $B=(c,d)$, $D=(a,d)$, and $C=(e,b)$, for some $a,b,c,d,e \in \bb{Z}_{n}$ (not necessarily distinct) such that $a\neq c$, $b\neq d$, $a \neq e$, and $e\neq c$. This forces $X=(c,b)$ and $Y=(e,d)$. Therefore $X$ is adjecent to $Y$, and thus $\subset N(A,B) \cup N(C,D)\cup \{B,C\}$ is a $6$ cycle. Therefore $|F| \leq 5$.

If $G$ is a forest with an induced copy of $P_{2}+P_{2}$, then a similar argument shows that $|G| \leq 5$ (by adding a vertex that induces the same $P_5$ that arises if we start with a $P_4$ as above). This completes the proof. 
\end{proof}

\subsection{Orthogonal Array Graphs}

We finally consider a family of strongly-regular graphs associated with orthogonal arrays. Let $m$ and $n$ be positive integers with $m<n+1$. An orthogonal array with parameters $(m, n)$ is an $m \times n^{2}$ array with entries in $\bb{Z}_n$ with the property that every $2\times n^{2}$ array consists of all $n^{2}$ possible pairs. Given an orthogonal array $\mathcal{O}$ with parameters $(m, n)$, by $X_{\mathcal{O}}$, we denote the graph on columns of $\mathcal{O}$, where two columns are adjacent if and only if there are no rows in which they have the same entry. We note that the graph $X_{\mathcal{O}}$ is the complement of the block graph of the orthogonal array $\mathcal{O}$. It is well known, see for example~\cite[Theorem 5.5.1]{GMbook}, that this is a strongly-regular graph with valency  $m(n-1)$ and least eigenvalue $m-n-1$. Application of the Delsarte-Hoffman ratio bound (Theorem~\ref{ratiobound}) shows that $\alpha(X_{\mathcal{O}})\leq n$. This bound is met by the set of columns of $\mathcal{O}$ whose first entry is $1$.

\begin{theorem}\label{OA}
Let $m, n$ be positive integers with $n>1+2m(m-1)$ and let $\mathcal{O}$ be an orthogonal array with parameters $(m, n)$. Then
\[
\tau(X_{\mathcal{O}})=n+1.
\]
Moreover, every maximum induced forest in $X_{m, n}$ is canonical. 
\end{theorem}
\begin{proof}
We now apply Lemma~\ref{lem:count} to characterize the maximum independent sets in $X_{\mathcal{O}}$. We note that $\eta(X_{\mathcal{O}})$ is the number of common neighbours of two non-adjacent vertices in the complement of $X_{\mathcal{O}}$. By \cite[Theorem 5.5.1]{GMbook}, we see that $\eta(X_{\mathcal{O}})=m(m-1)$. By Lemma~\ref{lem:count}, if $F$ is a non-canonical induced forest, we have $|F|\leq 2+ 2m(m-1)$. 

We can now conclude that if $\alpha(X_{\mathcal{O}})+1 =n+1 > 2+ 2m(m-1)$, then every maximum induced forest is canonical. 
\end{proof}

\section{Kneser graphs with non-canonical maximum forests}\label{large}

As noted in Subsection~\ref{subsec:kneser}, $K(2k, k)$ is a forest, so all of these graphs have non-canonical maximum forests.
The logical next family of Kneser graphs to consider are the graphs $K(2k+1,k)$, these graphs also have non-canonical maximum forests.

\begin{lemma}\label{lem:2k+1forest}
If $k >3$, the graph $K(2k+1,k)$ has a forest of order
\[
\binom {2k}{k} + 2k-2.
\]
hence the maximum forests are not canonical.
\end{lemma}
\begin{proof}
Let $F_1$ be the set of all vertices in $K(2k+1,k)$ that do not contain the element $2k+1$; $F_1$ is a set of $\frac12\binom{2k}{k} = \binom{2k-1}{k}$ disjoint edges.

For $i = 1, \dots, 2k-2$, define $x_i = \{i, i+1,\dots, i+k-3\}$ with the entries taken modulo $2k-1$.
Define the set $F_2$ of vertices of the form $\gamma_i = x_i \cup \{2k,2k+1\}$ with $i =1,\dots , 2k-2$. Clearly $F_2$ is a coclique and any vertex in $F_2$ is adjacent to at most one vertex in any edge of $F_1$ (specifically, the vertex that does not contain $2k$). Further, vertices $\gamma_i$ and $\gamma_j$, have exactly one common neighbour in $F_1$ if $j=i+1$, and no common neighbours otherwise.

Thus $F_1 \cup F_2$ forms a forest of order $\binom{2k}{k} + 2k-2$.
\end{proof}

The eigenvalue bound from Theorem~\ref{algbound} in this case is
\[
\tau(K(2k+1,k)) 
 < \frac{\binom{2k+1}{k} \binom{k}{k-1}}{\binom{k+1}{k}+\binom{k}{k-1} } +\frac{2\binom{2k+1}{k}}{\binom{k+1}{k}+\binom{k}{k-1}} 
= \frac{k+2}{k} \binom{2k}{k-1}. 
\]

This bound is larger than the forest given in Lemma~\ref{lem:2k+1forest}. We can do better using
the bound produced by Theorem~\ref{elembound}, which is 
$$\frac{\binom{2k+1}{k}\binom{k+1}{k}-2}{2\binom{k+1}{k}-2}=\frac{k+1}{2k}\binom{2k+1}{k}-\frac{1}{k}$$
but this is still significantly larger than the forest our construction produces.

The final case to consider is $K(7,3)$, and in this case Theorem~\ref{elembound} tells us that an induced forest has order at most 
\[
\frac{4}{6}\binom{7}{3}-\frac{1}{3} = \frac{2(35)-1}{3} = 23
\] 
which can be achieved by the forest consisting of all triples from $\{1,\dots, 6\}$ along with $\{1,2,7\}$, $\{1,3,7\}$ and $\{2,3,7\}$. 

\section{Further Work}\label{further}

It would be interesting to have more examples of graphs $G$ with $\alpha(G)$ very close to $\tau(G)$. We suspect that a characterization of the graphs with $\tau(G) = \alpha(G) +1$ is unlikely, but perhaps we can find properties of a graph that would imply these two values are close. In a sense, any such graph would have large independent sets that are uniformly connected to the vertices in its complement. Specifically, any two adjacent vertices outside of the large independent set would have to be adjacent to at least one common vertex in the independent set, and non-adjacent vertices to at least two. This may lead to some structure conditions on a graph that imply that $\tau(G) = \alpha(G) +1$.
We also suspect that focusing the search on strongly-regular graph may produce more interesting examples. 

All the examples of graphs we considered in this paper are graphs whose maximum independent sets have been characterized. Maximum independent sets in Paley graph on a square number vertices were characterized by Blokhius \cite{blokhuis1984subsets}. We will now discuss some computational results we obtained regarding induced forests in these graphs. Let $q$ be a power of an odd prime. Let $\bb{F}_q$ and $\bb{F}_{q^{2}}$ be a fields of cardinality $q$. By $\mathcal{P}(q^{2})$, we denote the Paley graph on $q^2$ vertices.  The vertex set for $\mathcal{P}(q^{2})$ is $\bb{F}_q^{2}$, and two vertices are adjacent if and only if their difference is a quadratic residue in the $\bb{F}_{q^{2}}$. It is well-known that the Paley graph is self-complementary. In this regard, we could consider the complement $\mathcal{P'}(q^2)$ of the $\mathcal{P}(q^2)$. We do so because the maximum independent sets in the complement have the following natural characterization.

\begin{theorem}(Blokhius \cite{blokhuis1984subsets})\label{Paleyind}
Let $q$ be a power of a prime and $S$ be the set of non-zero squares in $\bb{F}_{q^{2}}$, then $\alpha(\mathcal{P'}(q^{2}))=q$ and the set $\{s\bb{F}_{q}+e\ : s\in S \text{and}\ e\in \bb{F}_{q^{2}}\}$ is the set of all independent sets of size $q$.
\end{theorem} 

So the size of any canonical forest in $\mathcal{P'}(q^{2})$ is $q+1$. We will now use Theorem~\ref{algbound} to obtain an upper bound on the acyclic number. $\mathcal{P'}(q^{2})$ is strongly-regular graph and its spectrum is well known to be $(\frac{q^{2}-1}{2},\ \frac{q-1}{2},\ -\frac{q+1}{2})$ (see~\cite[Section 5.8]{GMbook}). Using Theorem~\ref{algbound}, we have 
\begin{align*}\label{pbound}
\tau(\mathcal{P}(q^{2}) < \dfrac{q^{2} (q^{2}+5)}{q^{2}+q} <  q+4
\end{align*}
 
In Example~\ref{exampleP9}, we concluded that $\tau(\mathcal{P'}(9))=5$. From the discussion above the size of a canonical forest in $\mathcal{P'}(9)$ is $4$ and thus in this case, maximum forests are not canonical. We will now consider two more Paley graphs of small order.

\begin{example}
Consider the graph $\mathcal{P'}(25)$, by Theorem~\ref{algbound}, a forest cannot have more than 8 vertices. We have $\bb{F}_{25} \cong \bb{F}_{5}[x]/\langle x^{2}+x+1\rangle$, and the set of quadratic residues is $S=\{\ov{0}\} \cup \{a,\ a\ov{x},\ a\left(\ov{x+1}\right)|\ a\in \bb{F}_{5}^{\ast} \}$. The induced subgraph $ \bb{F}_{5} \cup \{\ov{x+2}, \ \ov{x+4}\}$ is a forest (in fact, a tree) of order $7$ formed by adding two vertices to a maximum independent set. Since canonical forests have order $6$, this cannot be maximum forests. A computational search indicates that $7$ is the order of a maximum forest in this graph.
\end{example}

\begin{example}
Consider the complement of Paley graph on $49$ vertices. Again Theorem~\ref{algbound} implies a forest can have no more than $10$ vertices.  We have $\bb{F}_{49} \cong \bb{F}_{7}[x]/\langle x^{2}+1\rangle$. The set of quadratic residues is $S=\{\ov{0}\} \cup \{a,\ a\ov{x},\ a\left(\ov{x+1}\ \right),\ a\left(\ov{x-1}\ \right)|\ a\in \bb{F}_{7}^{\ast} \}$. The induced subgraph $ \bb{F}_{7} \cup \{\ov{x+2}, \ \ov{x+5}\}$ is a forest (in fact, a tree) of order $9$ formed by adding two vertices to a maximum independent set. Again computations indicate that $9$ is the order of a maximum forest in this graph, and canonical forests have order $8$.
\end{example}

In Example~\ref{exampleP9} and the above examples, maximum induced forests which are in fact trees were obtained by adding two vertices to a maximum independent set. Using Blokhius's characterization (Theorem~\ref{Paleyind}) of maximum independent sets, we used Sage~\cite{sage} to search if similar constructions were possible in bigger Paley graphs. We checked for all prime powers $7<q\leq 67$ that adding two vertices to a maximum independent set in $\mathcal{P'}(q^{2})$, will not result in a forest. So the examples we found may be anomalies occurring for small values of $q$. We make the following conjecture.

\begin{conj}\label{Conj:Paley}
For $q >7$ a prime power, $\tau(\mathcal{P}(q^2)) = q+1$.
\end{conj}
    
Paley graphs on $q$ vertices can be defined whenever $q$ is a prime power with $q\equiv 1 \pmod{4}$. Let $\mathcal{P'}(q)$ denote the graph on the field $\bb{F}_{q}$, in which two vertices are adjacent if and only if their difference is not a quadratic residue in $\bb{F}_{q}$. When $q$ is an even power, 
we conjectured above that $\tau(\mathcal{P'}(q)) = \sqrt{q}+1$. It is natural to ask the question of what happens when $p$ is not an even power of a prime. Applying Theorem~\ref{algbound}, we can show that $\tau(\mathcal{P'}(q))<\sqrt{q}+4$. 
 In this case, the order of the maximum independent sets is not known in general, but it is bounded by $\sqrt{q}$, and can be significantly smaller. For instance, when $q$ is a prime, \cite{hanson2021refined} shows that $\alpha(\mathcal{P'}(q))<\sqrt{\dfrac{q}{2}}+1$. From our computer searches it seems even in this case $\tau(\mathcal{P'}(q))$ is close to $\sqrt{q}$, so sometimes the induced forests are much larger than $\alpha(\mathcal{P'}(q))$. Further, $\tau(\mathcal{P'}(q))$ seems to be non-decreasing with $q$, which is not the case for the size of an independent set, and close to the eigenvalue bound. This may just be the case for small values of $q$, so more computational results would be helpful. A key missing result is a construction of an induced forest of size close to $\sqrt{q}$. Forests are bipartite graphs, and so existence of an induced forest of size $\sqrt{q}$ in $\mathcal{P'}(q)$ implies the existence of independent sets of size at least  $\sqrt{q}/2$. When $q$ is not an ever power of a prime, there are no known constructions of such large independent sets in $\mathcal{P'}(q)$. 

\bibliographystyle{plain}
\bibliography{ref}
\end{document}